\documentclass[12pt]{amsart}
\usepackage{amsfonts,amsmath,amssymb,amsthm}
\usepackage{mathrsfs}
\usepackage{graphics}
\usepackage{fullpage}
\usepackage{hyperref}
\usepackage[english]{babel}
\usepackage[latin1]{inputenc}
\usepackage{cite}
\allowdisplaybreaks
\numberwithin{equation}{section}

\newtheorem{theorem}{Theorem}[section]
\newtheorem{lemma}[theorem]{\bf Lemma}
\newtheorem{corollary}[theorem]{Corollary}
\newtheorem{prop}[theorem]{Proposition}
\newtheorem{definition}[theorem]{Definition}

\theoremstyle{remark}
\newtheorem{remark}[theorem]{Remark}

\newcommand{\N}{\mathbb{N}}
\newcommand{\Z}{\mathbb{Z}}

\newcommand{\R}{\mathbb{R}}

\newcommand{\calQ}{\mathcal{Q}}

\newcommand{\calS}{\mathcal{S}}
\newcommand{\calG}{\mathcal{G}}
\DeclareMathOperator{\supp}{supp}
\DeclareMathOperator{\dist}{dist}
\DeclareMathOperator{\esssup}{ess \; sup}
\DeclareMathOperator{\essinf}{ess \; inf}

\newcommand{\vf}{\mathbf{f}}
\newcommand{\vg}{\mathbf{g}}
\newcommand{\vh}{\mathbf{h}}
\newcommand{\vv}{\mathbf{v}}
\newcommand{\vu}{\mathbf{u}}
\newcommand{\ve}{\mathbf{e}}
\newcommand{\vk}{\mathbf{k}}
\newcommand{\calW}{\mathcal{W}}

\newcommand{\op}{\mathrm{op}}

\newcommand{\bk}{\backslash}

\def\Xint#1{\mathchoice
   {\XXint\displaystyle\textstyle{#1}}%
   {\XXint\textstyle\scriptstyle{#1}}%
   {\XXint\scriptstyle\scriptscriptstyle{#1}}%
   {\XXint\scriptscriptstyle\scriptscriptstyle{#1}}%
   \!\int}
\def\XXint#1#2#3{{\setbox0=\hbox{$#1{#2#3}{\int}$}
     \vcenter{\hbox{$#2#3$}}\kern-.5\wd0}}

\def\avgint{\Xint-}
\def\dashint{\Xint-}

\newcommand{\pp}{{p(\cdot)}}
\newcommand{\Pp}{\mathcal{P}}
\newcommand{\Lpp}{L^\pp}
\newcommand{\cpp}{{p'(\cdot)}}
\newcommand{\Lcpp}{L^\cpp}

\newcommand{\calA}{\mathcal{A}}
\newcommand{\A}{\mathcal{A}}

\newcommand{\scrWloc}{\mathscr{W}_{\text{loc}}^{1,1}}
\newcommand{\scrWpp}{\mathscr{W}^{1,\pp}}
\newcommand{\scrHpp}{\mathscr{H}^{1,\pp}}

\begin{document}

\title{Convolution Operators in Matrix Weighted, Variable Lebesgue Spaces}

\author{David Cruz-Uribe, OFS}

\address{Department of Mathematics, University of Alabama, Tuscaloosa, AL 35487, USA}

\author{Michael Penrod}

\address{Department of Mathematics, University of Alabama, Tuscaloosa, AL 35487, USA}

\thanks{The first author is partially supported by a
Simons Foundation Travel Support for Mathematicians Grant.\\
The authors would like to thank Michael Frazier and Scott Rodney for illuminating discussions on the density of smooth functions in matrix weighted Sobolev spaces.}



\date{\today}

\begin{abstract}
  We extend the theory of matrix weights to the variable Lebesgue
  spaces.  The theory of matrix $\calA_p$ weights has attracted
  considerable attention beginning with the work of Nazarov, Treil,
  and Volberg in the 1990s. We extend this theory by generalizing the
  matrix $\calA_p$ condition to the variable exponent setting.
We prove boundedness of the convolution operator
$\vf \mapsto \phi\ast \vf$ for $\phi \in C_c^\infty(\Omega)$, and show
that the approximate identity defined using $\phi$ converges in matrix
weighted, variable Lebesgue spaces $\Lpp(W,\Omega)$ for $W$ in matrix
$\calA_\pp$. Our approach to this problem is through averaging
operators; these results are of interest in their own right.  As an
application of our work, we prove a version of the classical $H=W$
theorem for matrix weighted, variable exponent Sobolev spaces.
\end{abstract}
\maketitle
\section{Introduction}

In this paper, we develop the theory of matrix weights on the variable
Lebesgue spaces.  Variable Lebesgue spaces have been extensively
studied since the early 1990s, both for their intrinsic interest and
for their applications to PDEs and the calculus of variations.
See~\cite{cruz-uribe_variable_2013,diening_lebesgue_2011} for a
detailed description of this history and further references.  Weighted
norm inequalities have been considered in the variable Lebesgue spaces
by several authors.
In~\cite{Cruz-Uribe-FiorenzaNeugebauer-2012,MR3572271,cruz-uribe_ofs_weighted_2020}
the first author and his collaborators extended the theory of (scalar)
Muckenhoupt $A_p$ weights to the variable exponent setting, defining
the class $A_\pp$.  Matrix weights have also been studied since the
1990s, first in the work of Nazarov, Treil and
Volberg~\cite{treil_wavelets_1997,nazarov_hunt_1996,MR1423034} and
later in the work of Christ and Goldberg, and
others~\cite{goldberg_matrix_2003,MR1813604,
  bownik_extrapolation_2022,isralowitz_matrix_2019,
  cruz-uribe_matrix_2016,MR2104276,roudenko_matrix-weighted_2002,MR3689742}.
In this work, they generalized the scalar theory of Muckenhoput $A_p$
weights to matrix $\A_p$ weights acting on vector-valued functions.

The goal of this paper is to continue the development of both these
areas by studying matrix weights in the variable Lebesgue spaces.  We
define a new class of variable exponent matrix weights $\A_\pp$ that
generalize both the scalar $A_\pp$ weights and the matrix $\A_p$
weights.  We then prove that convolution operators are are bounded on the
matrix weighted, variable Lebesgue spaces, and approximate identities
converge in norm.   As an application, we then exend the classical
result of Meyers and Serrin~\cite{meyers_h_1964} that $H=W$--i.e., 
that smooth functions are dense in the Sobolev
space $W^{1,p}(\Omega)$--to matrix weighted, variable exponent Sobolev
spaces. As a corollary, we prove that smooth functions with compact support are also dense in matrix weighted, variable exponent Sobolev spaces.

In order to state our main results, we first give some of the basic
definitions. An exponent function is a Lebesgue measurable
function $\pp: \Omega \to [1,\infty]$. Denote the collection of all
exponent functions on $\Omega$ by $\Pp(\Omega)$. Given a set $E$, let
$p_+(E) = \esssup_E p(x)$ and $p_-(E) =\essinf_E p(x)$. Let
$p_+=p_+(\Omega)$ and $p_-= p_-(\Omega)$.

%
%

An exponent function $\pp$ is locally log-H\"{o}lder continuous,
denoted by $\pp \in LH_0(\Omega)$, if there exists a constant $C_0$
such that for all $x,y\in \Omega$, $|x-y|<1/2$,
\begin{align*}
|p(x)-p(y)| \leq \frac{C_0}{-\log(|x-y|)}.
\end{align*}
We say that $\pp$ is log-H\"{o}lder continuous at infinity, denoted by
$\pp \in LH_\infty(\Omega)$, if there exist constants $C_\infty$ and
$p_\infty$ such that for all $x\in \Omega$,
\begin{align*}
|p(x)-p_\infty| \leq \frac{C_\infty}{\log(e+|x|)}.
\end{align*}
If $\pp$ is log-H\"{o}lder continuous locally and at infinity, we will denote this by  $\pp\in LH(\Omega)$. 

Given $\pp\in \Pp(\Omega)$, define the modular associated with $\pp$ by 
\[\rho_\pp(f) = \int_{\Omega\bk \Omega_\infty} |f(x)|^{p(x)} \, dx + \|f\|_{L^\infty(\Omega_\infty)},\]
where $\Omega_\infty = \{x\in \Omega : p(x) = \infty\}$. Define $\Lpp(\Omega)$ to be the collection of Lebesgue measurable functions $f:\Omega \to \R$ such that
\[\|f\|_{\Lpp(\Omega)} = \inf\{\lambda>0: \rho_\pp(f/\lambda)\leq 1\}<\infty.\]

A weight is a locally integrable function $w: \Omega \to \R$ such that
$w(x)>0$ a.e. Given a weight $w$, define $\Lpp(w,\Omega)$ as the
collection of measurable functions such that 
\[ \|f\|_{\Lpp(w,\Omega)} := \|fw\|_{\Lpp(\Omega)}<\infty.\]
Given $\pp \in \Pp(\Omega)$ we say that a weight $w$ is in $A_\pp(\Omega)$ if
\begin{equation} \label{eqn:scalar-App}
[w]_{A_\pp} = \sup_{Q\subset \Omega} \|w\chi_Q \|_{\Lpp(\Omega)}
  \| w^{-1}\chi_Q\|_{\Lcpp(\Omega)} < \infty.
\end{equation}

\begin{remark}
  When $\pp=p$ is constant, $1\leq p<\infty$, then \eqref{eqn:scalar-App}
  reduces to a condition equivalent to the classical Muckenhoupt $A_p$
  condition, after making the transformation $w\mapsto w^{1/p}$.  In
  both the scalar and matrix weighted case, it is often more
  convenient to rewrite the Muckenhoupt condition in this way:
  see~\cite{Cruz-Uribe-FiorenzaNeugebauer-2012,bownik_extrapolation_2022}
  for further information.
\end{remark}

Define $\Lpp(\Omega;\R^d)$ to be  the set of vector-valued, measurable functions $\vf : \Omega \to \R^d$ such that 
\[\| \vf\|_{\Lpp(\Omega;\R^d)} := \| |\vf|\|_{\Lpp(\Omega)} <\infty.\]

A matrix function maps a set $\Omega \subseteq \R^n$ to the collection
of $d\times d$ matrices. A matrix function is measurable if each of
its components is a measurable function. Let $\calS_d$ denote the
collection of $d\times d$ matrices that are self-adjoint and positive
semi-definite. A matrix weight is a measurable matrix function
$W : \Omega \to \calS_d$ such that
$|W|_{\op} \in L^1_{\text{loc}}(\Omega)$ (where $|\cdot|_\op$ denotes
the operator norm of $W$), or equivalently, the
eigenvalues of $W$ are locally integrable functions. A matrix weight
is invertible if it is positive definite almost everywhere.
Given a matrix weight $W:\Omega \to \calS_d$, define $\Lpp(W,\Omega)$
to be the collection of vector-valued, measurable functions
$\vf:\Omega \to \R^d$ satisfying
\[\|\vf\|_{\Lpp(W,\Omega)} := \| |W\vf| \|_{\Lpp(\Omega)}<\infty.\]
Given $\pp \in \Pp(\Omega)$ we say that a matrix weight $W$ is in $\A_\pp(\Omega)$ if
\begin{equation} \label{eqn:Matrix-App}
[W]_{\calA_\pp} = \sup_{Q}|Q|^{-1} \big\|  \| |W(x)W^{-1}(\cdot)|_{\op}\chi_Q(\cdot)\|_{\Lcpp(\Omega)} \chi_Q(x)\big\|_{\Lpp_x(\Omega)}<\infty.
\end{equation}
Here, $\Lpp_x$ indicates that the integration in the norm is done with
respect to the $x$ variable.  As is the case for scalar weights, if
$\Omega = \R^n$, then we get an equivalent definition if we replace
cubes with balls. We will want to elide between balls and cubes on
more general domains. Therefore, given any matrix weight $W$ on a
domain $\Omega$, we will implicitly assume that it satisfies the
matrix $\calA_\pp$ condition on some larger domain $\Omega'$ and we
will suppress any reference to the domain, writing $\calA_\pp$ instead
of $\calA_\pp(\Omega)$. The problem of characterizing those domains
$\Omega$ such that every $W \in \calA_\pp(\Omega)$ is the restriction
of a matrix in $\calA_\pp(\R^n)$ is open.

\begin{remark}
Again when $\pp=p$, $1\leq p<\infty$, \eqref{eqn:Matrix-App} reduces
to the definition of matrix $\A_p$ given by
Roudenko~\cite{roudenko_matrix-weighted_2002}, $1<p<\infty$, and
Frazier and Roudenko~\cite{MR2104276}, $p=1$, with the transformation
$W\mapsto W^{1/p}$.
\end{remark}

\medskip

Our first main result is on the boundedness of convolution operators
and the convergence of approximate identities on $\Lpp(W,\Omega)$.
Let $\phi \in C_c^\infty(B(0,1))$ be a nonnegative, radially symmetric
and decreasing function with $\int_{\R^n} \phi(x)\, dx = 1$. Define
the approximate identity $\{\phi_t\}_{t>0}$ by
$\phi_t(x) = t^{-n} \phi(x/t)$.  It is well-known that $\phi_t \ast f$
converges to $f$ in $L^p(\Omega)$ as $t\rightarrow 0$,  $1\leq p <\infty$, and the same is
true in $L^p(w,\Omega)$ if $w$ is in the Muckenhoupt class $A_p$.
(See~\cite{turesson00}.)  In
\cite{cruz-uribe_matrix_2016}, the first author, Moen, and Rodney extended
this to matrix weighted spaces $L^p(W,\Omega)$ for weights
$W \in \calA_p$, $1\leq p<\infty$.  Here we prove the analogous result for
$\Lpp(W,\Omega)$.

\begin{theorem}\label{ConvolutionConvergence:theorem}
  Given $\pp \in \Pp(\Omega)$, suppose $1\leq p_- \leq p_+<\infty$ and
  $\pp\in LH(\Omega)$.  Fix  $W \in \calA_\pp$.   Let
  $\phi \in C_c^\infty(B(0,1))$ be a nonnegative, radially symmetric and
  decreasing function with $\int_{\R^n} \phi(x) \,dx=1$. For $t>0$, let
  $\phi_t(x) = t^{-n}\phi(x/t)$. Then for every $\vf\in \Lpp(W,\Omega)$,
\begin{equation}\label{ConvolutionBound}
  \sup_{t>0} \|\phi_t\ast \vf\|_{\Lpp(W,\Omega)}
  \leq C(n,\pp) [W]_{\calA_\pp} \|\vf\|_{\Lpp(W,\Omega)}.
\end{equation}
Moreover, we have that for every $\vf\in \Lpp(W,\Omega)$,
\begin{align}\label{ConvolutionConvergence:ineq}
\lim_{t\to 0} \|\phi_t \ast \vf - \vf\|_{\Lpp(W,\Omega)} = 0.
\end{align}
\end{theorem}

In the scalar case, Theorem~\eqref{ConvolutionBound} is usually proved
(for $p>1$)
by using the fact that $|\phi_t*f|$ is dominated pointwise by $Mf$,
where $M$ is the Hardy-Littlewood maximal operator.  However, the
maximal operator needs to be redefined to make sense in the matrix
weighted spaces (see~\cite{goldberg_matrix_2003,MR1813604}) and so is
not immediately applicable.
Therefore, we instead follow the approach introduced
in~\cite{cruz-uribe_matrix_2016} and first develop the theory of matrix
weighted inequalities for averaging operators, which we then apply to
estimate convolution operators.  We think that these results are of
independent interest.

\medskip

Our second main result is an application of Theorem
\ref{ConvolutionConvergence:theorem} to matrix weighted, vector-valued
Sobolev spaces. We briefly sketch the requisite definitions and defer
stating the complete definitions to Section~\ref{H=W:sec}. Given $\Omega \subseteq \R^n$, $\pp \in \Pp(\Omega)$, and $W : \Omega \to \calS_d$, the space
$\scrWpp(W,\Omega)$ is the collection of vector-valued functions
$\vf = (f_1, \ldots, f_d)^T \in \Lpp(W,\Omega)$ such that $|W D\vf|_{\op}\in \Lpp(\Omega)$, where $D\vf$ is the Jacobian matrix of $\vf$, i.e., a $d\times n$ matrix where $(D\vf)_{i,j} = \partial_j f_i$ for $i=1,\ldots, d$ and $j=1,\ldots, n$. The space $\scrHpp(W,\Omega)$ is the closure of
$C^\infty(\Omega;\R^d) \cap \scrWpp(W,\Omega)$.

\begin{theorem}\label{H=W:theorem}
Given $\pp \in \Pp(\Omega)$, suppose $1\leq p_-\leq p_+<\infty$ and
$\pp\in LH(\Omega)$.  Given $W\in \calA_\pp$,  $\scrWpp(W,\Omega) = \scrHpp(W,\Omega)$.
\end{theorem}

\begin{remark}
In~\cite{isralowitz_matrix_2019}, as part of their study of
  solutions to systems of degenerate elliptic questions, the authors
  stated a version of  Theorem~\ref{H=W:theorem} when $\pp=p$, $1\leq p <\infty$, without proof.  
\end{remark}

\medskip

We also consider the scalar Sobolev space $\scrWpp(v,W,\Omega)$, where
$v=|W|_\op$, defined to be the collection of $f\in \Lpp(v, \Omega)$ such that $\nabla f\in \Lpp(W,\Omega)$. Let $\scrHpp(v,W,\Omega)$ denote the closure of
$C^\infty(\Omega)$ in $\scrWpp(v,W,\Omega)$.  

\begin{theorem}\label{H=W-scalar:theorem}
Given $\pp \in \Pp(\Omega)$, suppose that  $1\leq p_-\leq p_+<\infty$ and
$\pp\in LH(\Omega)$. Given an $n\times n$ matrix weight $W\in \calA_\pp$, $\scrWpp(v,W,\Omega) = \scrHpp(v,W,\Omega)$.
\end{theorem}

\begin{remark}
  When $\pp=p$, $1\leq p <\infty$, Theorem~\ref{H=W-scalar:theorem}
  was proved in~\cite{cruz-uribe_matrix_2016}.
\end{remark}

As a corollary to these two results, we get the following.

\begin{corollary} \label{cor:smooth-dense-Rn}
  Given $\pp \in \Pp(\R^n)$, suppose that $1\leq p_-\leq p_+<\infty$ and
$\pp\in LH(\R^n)$.  Given a $d\times d$ matrix weight $W\in \calA_\pp$, $C_c^\infty(\R^n;\R^d)$
is dense in $\scrWpp(W,\R^n)$; similarly, given an $n\times n$ matrix weight $W\in \A_\pp$,  $C_c^\infty(\R^n)$ is dense in
$\scrWpp(v,W,\R^n)$. 
\end{corollary}

\medskip

The remainder of this paper is organized as follows. In
Section~\ref{Preliminaries:sec}, we state some results about variable
Lebesgue spaces needed in our proofs. In
Section~\ref{MatrixWeights:sec}, we prove several results about matrix
weighted spaces. In Section~\ref{MatrixApp}, we characterize matrix
$\A_\pp$ in terms of averaging operators;  we also show a
connection between matrix $\A_\pp$ weights and scalar $A_\pp$ weights
that we use in the subsequent sections. In
Section~\ref{ConvConvergence:sec}, we prove
Theorem~\ref{ConvolutionConvergence:theorem}. Finally, in Section
\ref{H=W:sec}, we prove Theorems~\ref{H=W:theorem}
and~\ref{H=W-scalar:theorem} and Corollary~\ref{cor:smooth-dense-Rn}.

Throughout this paper, we will use the following notation. The symbol
$n$ will always denote the dimension of the Euclidean space $\R^n$ and
we denote the usual Euclidean norm by $|\cdot|$. We will use $d$ to
denote the dimension of matrix and vector-valued functions. We will
take the domain of our functions to be a Lebesgue measurable set
$\Omega\subseteq \R^n$. The set $\Omega$ need not, \textit{a priori},
be bounded. $C$ will denote a constant that may vary in value from
line to line and which will depend on underlying parameters. If we
want to specify the dependence, we will write, for instance,
$C(d)$. If the value of the constant is not important, we will often
write $A \lesssim B$ instead of $A \leq cB$ for some constant $c$. We
will also use the convention that $1/\infty=0$.
\section{Preliminaries}\label{Preliminaries:sec}

In this section, we state some important lemmas and propositions about
variable Lebesgue spaces. We refer to
\cite{cruz-uribe_variable_2013} and \cite{diening_lebesgue_2011} for
proofs and more information about these spaces. We begin with H\"{o}lder's
inequality. Given $\pp \in \Pp(\Omega)$, let $\cpp$ be the conjugate exponent function to $\pp$
defined by $\frac{1}{p(x)}+\frac{1}{p'(x)}=1$, $x\in \Omega$.
\begin{lemma}\cite[Theorem 2.26]{cruz-uribe_variable_2013}\label{Holder}
Given $\pp \in \Pp(\Omega)$, if $f \in \Lpp(\Omega)$ and $g \in
\Lcpp(\Omega)$, then $fg \in L^1(\Omega)$ and 
\[\int_\Omega |f(x) g(x) |dx \leq 4 \|f\|_{\Lpp}\|g\|_{\Lcpp}.\]
\end{lemma}

In our hypotheses, we assume log-H\"{o}lder continuity primarily
because it implies the so-called property $\calG$, which allows us to
assemble the norm of a function from its pieces over disjoint
collections of cubes. Berezhnoi~\cite{Berezhnoi-1999} introduced this condition 
to prove results about the boundedness of the
Hardy-Littlewood maximal operator between ideal Banach
spaces. Kopaliani~\cite{Kopaliani-2010} first considered this condition in the variable Lebesgue spaces.
Diening~\cite{diening_lebesgue_2011}
used property $\calG$ to study maximal operators defined using variable
Lebesgue space norms.  A bilinear version of property $\calG$ played a
role in~\cite{cruz-uribe_ofs_weighted_2020}.  
\begin{definition}\label{PropertyG:def}
An exponent $\pp \in \Pp(\Omega)$ satisfies property $\calG$ on
$\Omega$, denoted $\pp \in \calG(\Omega)$, if there exists a constant
$C$ such that for all $f \in \Lpp(\Omega)$, $g\in \Lcpp(\Omega)$, and all pairwise disjoint collections $\calQ$ of cubes $Q\subseteq \Omega$,
\[ \sum_{Q\in \calQ} \|\chi_Q f\|_{\Lpp(\Omega)} \|\chi_Q g\|_{\Lcpp(\Omega)} \leq C \|f\|_{\Lpp(\Omega)} \|g\|_{\Lcpp(\Omega)}.\]
The smallest constant $C$ is called the $\calG$-constant of $\pp$ on $\Omega$. 
\end{definition}
\begin{lemma}\label{LH-G:lemma}
If $\pp\in LH(\Omega)$, then $\pp \in \calG(\Omega)$ and the $\calG$-constant depends only on $n$ and the log-H\"{o}lder constants of $\pp$.
\end{lemma}

Lemma \ref{LH-G:lemma} is proved in~\cite{diening_lebesgue_2011} for
$\Omega = \R^n$. We extend their lemma to any set
$\Omega \subseteq \R^n$ using the following extension property of
log-H\"{o}lder continuous exponent functions.
\begin{lemma}\cite[Lemma 2.4]{cruz-uribe_variable_2013}\label{Extension:lemma}
  Given a set $\Omega \subseteq \R^n$ and $\pp \in LH(\Omega)$, there
  exists a function $\tilde{p}(\cdot)\in LH(\R^n)$ such that
  $\tilde{p}(x) = p(x)$ for $x\in \Omega$, $\tilde{p}_- = p_-$ and
  $\tilde{p}_+ = p_+$.
\end{lemma}
\begin{proof}[Proof of Lemma \ref{LH-G:lemma}]
 Fix $\Omega \subset \R^n$ and $\pp \in LH(\Omega)$. By Lemma
 \ref{Extension:lemma}, we can extend $\pp$ to an exponent function in 
 $LH(\R^n)$. By
  \cite[Theorem 7.3.22]{diening_lebesgue_2011},
  $\pp\in \calG(\R^n)$. Let $\calQ$ be a pairwise
  disjoint collection of cubes in $\Omega$. Let $f \in \Lpp(\Omega)$
  and $g\in \Lcpp(\Omega)$.  Extend them to all of $\R^n$ by setting
  them equal to $0$ on $\R^n\setminus \Omega$. Then
\begin{multline*}
  \sum_{Q\in \calQ} \|\chi_Q f\|_{\Lpp(\Omega)}\|\chi_Q
  g\|_{\Lcpp(\Omega)}
  = \sum_{Q\in \calQ} \|\chi_Q
  {f}\|_{L^{{p}(\cdot)}(\R^n)}
  \|\chi_Q {g}\|_{L^{{p}'(\cdot)}(\R^n)}\\
  \leq C \|{f}\|_{L^{{p}(\cdot)}(\R^n)}
  \|{g}\|_{L^{{p}'(\cdot)}(\R^n)}
  = C \|f\|_{\Lpp(\Omega)}
  \|g\|_{\Lcpp(\Omega)}.
\end{multline*}
\end{proof}

%
%
%
\section{Matrix Weighted Spaces}\label{MatrixWeights:sec}
In this section, we prove some basic properties about matrix weighted
spaces.  In working with vector-valued functions $\vf : \Omega \to
\R^d$, it will sometimes be convenient to use an equivalent norm on
$\R^d$.  Define $ |x|_1 = \sum_{i=1}^d |x_i| $;
then
\begin{equation} \label{eqn:equiv-norms}
  |x|\leq |x|_1 \leq \sqrt{d}\,|x|.
\end{equation}
It follows from this equivalence that for any finite set of vector-valued functions $\{\vf_i\}_{i=1}^m \subset \Lpp(\Omega;\R^d)$, we have
\begin{equation}\label{SumNorm:equiv}
  \frac{1}{m}\sum_{i=1}^m \|\vf_i\|_{\Lpp(\Omega)}
  \leq   \left\|\sum_{i=1}^m |\vf_i|\right\|_{\Lpp(\Omega)}
  \leq \sum_{i=1}^m \|\vf_i\|_{\Lpp(\Omega)}.
\end{equation}
The upper bound is immediate from the triangle inequality. The lower
bound holds since for each $i$, $|\vf_i(x)| \leq \sum_{j=1}^m
|\vf_j(x)|$.

Recall that the operator norm of a $d\times d$ matrix is defined by
\[|A|_{\op} = \sup_{\substack{\vv\in \R^d\\ |\vv|=1}} |A\vv|.\]

If $W \in \calS_d$, then it has $d$ non-negative eigenvalues, $\lambda_i$, $1\leq i \leq d$. Moreover, there exists an orthogonal matrix $U$ such that $U^TWU$ is diagonal. We denote a diagonal matrix by $D(\lambda_1, \ldots, \lambda_d)=D(\lambda_i)$. If $W$ is a measurable matrix function with values in $\calS_d$, then we can choose the matrix function $U$ to be measurable.
\begin{lemma}{\cite[Lemma 2.3.5]{ron_frames_1995}}\label{Diagonal}
Given a matrix function $W:\Omega \to \calS_d$, there exists a $d\times d$ measurable matrix function $U$ defined on $\Omega$ such that $U(x)$ is orthogonal and $U^T(x)W(x)U(x)$ is diagonal for a.e. $x\in \Omega$.
\end{lemma}

It follows from \cite[Theorem 2.34]{cruz-uribe_variable_2013} that given $\pp\in \Pp(\Omega)$, for any $f\in \Lpp(\Omega)$, there exists $g\in \Lcpp(\Omega)$ with $\|g\|_{\Lcpp(\Omega)} \leq 1$ such that 
\begin{align}\label{ScalarDuality}
\|f\|_{\Lpp(\Omega)} \leq 4 \int_{\Omega} f(x) g(x) \, dx.
\end{align}
We can extend this to the vector-valued function setting.
\begin{prop}\label{VectorDuality}
Let $\pp\in \Pp(\Omega)$. Then for any $\vf \in \Lpp(\Omega;\R^d)$, there exists $\vg\in \Lcpp(\Omega; \R^d)$ with $\|\vg\|_{\Lcpp(\Omega;\R^d)}\leq d$ such that
\[ \|\vf\|_{\Lpp(\Omega;\R^d)} \leq 4 \int_{\Omega} \vf(x) \cdot \vg(x) \, dx.\]
\end{prop}
\begin{proof}
Let $\vf=(f_1, \ldots, f_d)^T \in \Lpp(\Omega;\R^d)$.
By~\eqref{eqn:equiv-norms} and the triangle inequality,
\[ \|\vf\|_{\Lpp(\Omega;\R^d)} \leq \||\vf|_1\|_{\Lpp(\Omega;\R^d)} \leq  \sum_{i=1}^d \|f_i\|_{\Lpp(\Omega)}.\]

By inequality \eqref{ScalarDuality}, for each $i$ there exists $g_i \in \Lcpp(\Omega)$ with $\|g_i\|_{\Lcpp(\Omega)} \leq 1$ such that $\|f_i\|_{\Lpp(\Omega)} \leq 4 \int_{\Omega} f_i(x)g_i(x) \, dx$. Thus, the last expression is bounded by
\[ 4\sum_{i=1}^d \int_{\Omega} f_i(x) g_i(x) \, dx  = 4\int_{\Omega} \vf(x) \cdot \vg(x) \,dx,\]
where $\vg= (g_1, \ldots, g_d)^T$. Again by~\eqref{eqn:equiv-norms} and the triangle inequality, 
\[\|\vg\|_{\Lcpp(\Omega;\R^d)}
  \leq \||\vg|_1\|_{\Lcpp(\Omega;\R^d)}
  \leq  \sum_{i=1}^d \|g_i\|_{\Lcpp(\Omega)}\leq  d.\]
\end{proof}
We will also need two facts about  the operator norm of matrices.
First, a lemma.
\begin{lemma}{\cite[Lemma 3.2]{roudenko_matrix-weighted_2002}}\label{opNorm:equiv}
If $\{\ve_1, \ldots, \ve_n\}$ is any orthonormal basis in $\R^n$, then for any $d\times n$ matrix $V$, we have
\begin{align*}
|V|_{\op} \approx \sum_{i=1}^n |V \ve_i|,
\end{align*}
with implicit constants depending only on $n$.
\end{lemma}

Second, recall that for any two self-adjoint $d\times d$ matrices $V$ and $W$, we have
\begin{align*}
|VW|_{\op} = |(VW)^*|_{\op} = |W^*V^*|_{\op} = |WV|_{\op}.
\end{align*}

Dense subspaces play a key role in the proof of
Theorems~\ref{H=W:theorem} and~\ref{H=W-scalar:theorem}. In the case
that $\pp=p$, $1\leq p<\infty$, $L_c^\infty(\Omega; \R^d)$ and $C_c^\infty(\Omega;\R^d)$ are dense in $L^p(W,\Omega)$ (see \cite[Section 3]{cruz-uribe_matrix_2016}). We extend these results to the variable exponent setting.
\begin{prop}\label{BddCompSuppDense}
Given a matrix weight $W : \Omega \to \calS_d$ and $\pp\in \Pp(\Omega)$ with $p_+<\infty$, $L_c^\infty(\Omega;\R^d)$ is dense in $\Lpp(W,\Omega)$.
\end{prop}

\begin{proof}
First assume $W$ is diagonal, i.e., $W(\cdot)=D(\lambda_i(\cdot))$. Fix $\vf = (f_1 , \ldots, f_d)^T\in \Lpp(W,\Omega)$. For $x\in \Omega$, since $W(x) \in \calS_d$, $\lambda_i(x)\geq 0$ for $i=1,\ldots, d$. By the equivalence of norms \eqref{SumNorm:equiv}, we have that 
\[\|W\vf\|_{\Lpp(\Omega;\R^d)} = \left\| \left|\sum_{i=1}^d \lambda_i \vf \cdot \ve_i\right|\right\|_{\Lpp(\Omega)} 
\leq \left\| \sum_{i=1}^d \lambda_i |f_i| \right\|_{\Lpp(\Omega)}
 \approx \sum_{i=1}^d \| \lambda_i f_i\|_{\Lpp(\Omega)},\]
where $\{\ve_i\}_{i=1}^d$ is the standard coordinate basis in
$\R^d$. In \cite[Lemma 3.1]{MR3572271}, the authors prove
$L_c^\infty(\Omega)$ is dense in $\Lpp(w,\Omega)$ for any weight
$w$. Let $\epsilon >0$. For each $i=1,\ldots, d$, choose
$g_i \in L_c^\infty(\Omega)$ such that
$\|f_i - g_i\|_{\Lpp(\lambda_i,\Omega)}<\epsilon/d$. Define
$\vg = (g_1, \ldots, g_d)^T$. Since $g_i \in L_c^\infty(\Omega)$,
$\vg \in L_c^\infty(\Omega;\R^d)$, and so
\[\| \vf-\vg\|_{\Lpp(W,\Omega)} \leq \sum_{i=1}^d \| \lambda_i (f_i - g_i)\|_{\Lpp(\Omega)}<\epsilon.\]
Thus, $L_c^\infty(\Omega;\R^d)$ is dense $\Lpp(W,\Omega)$. This
completes the case where $W$ is diagonal.

Now fix an arbitrary matrix weight $W: \Omega \to \calS_d$. By Lemma \ref{Diagonal}, there exists a measurable orthogonal matrix function $U$ such that $D=U^T W U$ is diagonal. Let $\vf\in \Lpp(W,\Omega)$ and set $\vh=U^T \vf$. Then by orthogonality of $U$, we have that 
\[|D\vh| = |U^TWUU^T\vf| = |W\vf|.\]
Hence $\vh \in \Lpp(D,\Omega)$, and so by the previous argument, $L^\infty_c(\Omega;\R^d)$ is dense in $\Lpp(D,\Omega)$. Thus, for any $\epsilon>0$, there exists $\vg\in L_c^\infty(\Omega;\R^d)$ such that $\|\vh - \vg\|_{\Lpp(D,\Omega)} < \epsilon$. Using orthogonality again, we have that
\[|D(\vh - \vg)| =|U^T W U (U^T\vf - \vg)| = |W(\vf - U \vg)|.\]
Since $U$ is orthogonal, $|U\vg| =|\vg|$. Thus, $U\vg \in L_c^\infty(\Omega;\R^d)$. This completes the proof.
\end{proof}

\begin{prop}\label{ContCompSuppDense}
Given a matrix weight $W: \Omega \to \calS_d$ and $\pp\in \Pp(\Omega)$ with $p_+<\infty$, $C_c^\infty(\Omega;\R^d)$ is dense in $\Lpp(W,\Omega)$. 
\end{prop}
\begin{proof}

Fix $\vf = (f_1, \ldots, f_d)^T\in \Lpp(W,\Omega)$ and let $\epsilon>0$. By Proposition \ref{BddCompSuppDense}, there exists $\vg\in L_c^\infty(\Omega;\R^d)$ such that $\|\vf-\vg\|_{\Lpp(W,\Omega)} < \epsilon/2$. 

Choose an open ball $B$ such that $\supp(\vg)\subset B\cap\Omega$. Since $p_+<\infty$, $C_c^\infty(\Omega ; \R^d)$ is dense in $L^{p_+}(W,\Omega)$ (See \cite[Proposition 3.7]{cruz-uribe_matrix_2016}). Choose $\vh \in C_c^\infty(\Omega;\R^d)$ such that 
\[\|\vg-\vh\|_{L^{p_+}(W,\Omega)}\leq \frac{\epsilon}{2(1+|B\cap\Omega|)}.\]

The rest of the proof follows from \cite[Corollary 2.48]{cruz-uribe_variable_2013}. Since $|B\cap\Omega|<\infty$, 
\[\|\vg-\vh\|_{\Lpp(W,\Omega)} \leq (1+|B\cap\Omega|)\|\vg-\vh\|_{L^{p_+}(W,\Omega)}.\]
Now we have that
\begin{multline*}
\|\vf-\vh\|_{\Lpp(W,\Omega)} \leq \|\vf-\vg\|_{\Lpp(W,\Omega)} + \|\vg-\vh\|_{\Lpp(W,\Omega)}\\
	 \leq \|\vf-\vg\|_{\Lpp(W,\Omega)} + (1+|B\cap\Omega|)\|\vg-\vh\|_{L^{p_+}(W,\Omega)}
	 < \epsilon.
\end{multline*}
Thus, $C_c^\infty(\Omega;\R^d)$ is dense in $\Lpp(W,\Omega)$.
\end{proof}

%
\section{Matrix $\calA_\pp$}\label{MatrixApp}

In this section, we give two related characterizations of matrix $\A_\pp$ weights in terms of
averaging operators.   Given a cube $Q\subseteq \Omega$, define the averaging operator $A_Q$ by 
\[A_Q \vf(x) = \dashint_{Q} \vf(y) \,dy \; \chi_{Q}(x).\]

\begin{theorem}\label{AvgOpBddEquivApp}
  Let $\pp \in \Pp(\Omega)$. Then a matrix weight
  $W: \Omega \to \calS_d$ satisfies $W \in \calA_\pp$ if and only if
  $A_Q : \Lpp(W,\Omega)\to \Lpp(W,\Omega)$ uniformly for all cubes
  $Q \subseteq \Omega$. Moreover,
\[ \|A_Q\| \leq 4[W]_{\calA_{\pp}} \leq C(d) \|A_Q\|,\]
where $\|A_Q\|$ denotes the operator norm of $A_Q$.
\end{theorem}

Similarly, given any collection $\calQ$ of pairwise disjoint cubes
$Q\subseteq \Omega$, define the averaging operator $A_\calQ$ by
\[A_\calQ \vf(x) = \sum_{Q\in\calQ} \dashint_Q \vf(y) \, dy \; \chi_Q(x).\]
For our second characterization, we must assume the exponent function
satisfies some regularity conditions.  

\begin{theorem}\label{SumAvgOpBdd}
  Given $\pp \in \Pp(\Omega)$, suppose $\pp \in LH(\Omega)$.  Then a matrix
  weight $W:\Omega \to \calS_d$ satisfies $W \in \calA_\pp$ if and only if it satisfies
\begin{equation}\label{SumAvgOpBdd:ineq}
 \|A_\calQ \vf\|_{\Lpp(W,\Omega)} \leq C(n,d,\pp) [W]_{\calA_\pp} \|\vf\|_{\Lpp(W,\Omega)}.
\end{equation}
\end{theorem}

\begin{remark}
  In the proof of Theorem~\ref{SumAvgOpBdd} we only use that $\pp\in
  LH(\Omega)$ to apply Lemma~\ref{LH-G:lemma} to show that $\pp$
  satisfies property $\calG$.    Thus our results hold if we only assume this
  property, which is strictly weaker than log-H\"older continuity:
  see~\cite[Remark~7.3.8]{diening_lebesgue_2011}. 
\end{remark}

\medskip

We will prove each of these results in turn.  To prove
Theorem~\ref{AvgOpBddEquivApp} we need to give an equivalent
characterization of matrix $\A_\pp$, based on the notion of reducing
operators.  To define these, we need to consider general norms on the
vector space $\R^d$.  
\begin{prop}{\cite[Theorem 4.11]{bownik_extrapolation_2022}}\label{EllipsoidApprox}
  Given any measurable norm function
  $r: \R^n \times \R^d \to [0,\infty)$, there exists a measurable
  matrix function $W: \Omega \to \calS_d$ such that for any
  $x\in \R^n$ and $\vv\in \R^d$,
\[r(x,\vv) \leq |W(x) \vv|\leq \sqrt{d}\,r(x,\vv).\]
\end{prop}

\begin{remark}
  Proposition~\ref{EllipsoidApprox} was first proved
  in~\cite[Proposition 1.2]{goldberg_matrix_2003}; however, this paper
  did not consider the question of the measurability of $W$; this was
  justified in~\cite{bownik_extrapolation_2022}.
\end{remark}

Given a matrix weight $W$ and a cube $Q\subset \Omega$, we now define
the associated reducing operator, which can be thought of as an
$\Lpp(\Omega)$ average of $W$ on $Q$.
For each $x \in \Omega$, define the norm 
$r(x,\cdot): \R^d \to [0,\infty)$ by $r(x,\vv) = |W(x)\vv|$. Let
$r^*(x,\cdot)$ be the dual norm to $r$, given by
$r^*(x,\vv)= |W^{-1}(x) \vv|$. Given a cube $Q$, define
$\langle r\rangle_{\pp,Q}: \R^d \to [0,\infty)$ 
by
\begin{equation}\label{AvgNorm}
  \langle r\rangle_{\pp,Q} (\vv)
  = |Q|^{-1/p_Q} \|r(\cdot,\vv)\chi_Q(\cdot)\|_{\Lpp(\Omega)}
  = |Q|^{-1/p_Q} \| \vv\chi_Q(\cdot)\|_{\Lpp(W,\Omega)},
\end{equation}
where $p_Q$ is the harmonic mean of $\pp$ on $Q$, defined by
\[ \frac{1}{p_Q} = \avgint_Q \frac{dx}{p(x)}. \]
Since $\|\cdot\|_{\Lpp(W,\Omega)}$ is a norm, it follows that
$\langle r \rangle_{\pp,Q}$ is a norm on $\R^d$. Hence, by Proposition
\ref{EllipsoidApprox}, there exists a positive-definite, self-adjoint
(constant) matrix $\calW_Q^\pp$ such that
$ \langle r\rangle_{\pp, Q} (\vv) \approx |\calW_Q^\pp \vv|$ for all
$\vv\in \R^d$. The matrix $\calW_Q^\pp$ is referred to as the reducing
operator associated to $r$ on $Q$. We will denote the reducing
operator associated to $r^*$ on $Q$ by $\overline{\calW}_Q^\cpp$, i.e., $\langle r^*\rangle_{\cpp,Q} (\vv)\approx |\overline{\calW}_Q^\cpp \vv|$.
\begin{remark}
If we use the definition of the $L^\pp$ average of a function
from~\cite{diening_lebesgue_2011}, we could alternatively define $\langle r\rangle_{\pp, Q}$ by
 \[  \langle r\rangle_{\pp,Q} (\vv)
   = \frac{\|r(\cdot,\vv)\chi_Q(\cdot)\|_{\Lpp(\Omega)}}
   {\|\chi_Q\|_{\Lpp(\Omega)}}. \]
 However, if $\pp \in LH(\Omega)$, then
 $\| \chi_Q\|_{\Lpp(\Omega)} \approx |Q|^{1/p_Q}$, and the two definitions
 of $\langle r \rangle_{\pp,Q}$ are equivalent.  In fact, for this equivalence it would suffice to
 assume the weaker condition that $\pp \in K_0(\Omega)$.  We refer the
 interested reader to~\cite[Chapter 4]{cruz-uribe_variable_2013} for
 more information.
\end{remark}

We now use reducing operators to give an equivalent characterization of $\calA_\pp$. 
\begin{prop}\label{ApdotReducingOps}
  Let $\pp \in \Pp(\Omega)$ and $W:\Omega \to \calS_d$ be a matrix
  weight. Then $W \in \calA_\pp$ if and only if
\begin{align}\label{ApdotReducingOps:ineq}
[W]_{\calA_\pp}^R = \sup_{Q\subset \Omega} |\calW_Q^\pp \overline{\calW}_Q^\cpp|_{\op} <\infty.
\end{align}
Moreover, $[W]_{\calA_\pp}^R \approx [W]_{\calA_\pp}$, with implicit constants depending only on $d$.
\end{prop}
\begin{proof}
  Let $Q\subseteq \Omega$ be a cube and $\{\ve_i\}_{i=1}^d$ be the
  standard coordinate basis in $\R^d$. By Lemma~\ref{opNorm:equiv},
  inequality \eqref{SumNorm:equiv}, and the definition of the reducing
  operator $\overline{\calW}_Q^\cpp$, we have
\begin{align*}
  &|Q|^{-1}\big\|\||W(x)W^{-1}(\cdot)|_{\op}\chi_Q(\cdot)\|_{\Lpp(\Omega)}
    \chi_Q(x) \big\|_{\Lcpp_x(\Omega)}\\
  &\qquad 	 \approx |Q|^{-1} \bigg\|
    \bigg\|\sum_{i=1}^d |W^{-1}(\cdot) W(x)
    \ve_i|\chi_Q(\cdot)\bigg\|_{\Lcpp(\Omega)}
    \chi_Q(x)\bigg\|_{\Lpp_x(\Omega)}\\
  &\qquad 	 \approx \sum_{i=1}^d|Q|^{-1}
    \big\|\||W^{-1}(\cdot)W(x)\ve_i|\chi_Q(\cdot)\|_{\Lcpp(\Omega)}
    \chi_Q(x)\big\|_{\Lpp_x(\Omega)}\\
	&\qquad	\approx \sum_{i=1}^d |Q|^{-1/p_Q}
   \||\overline{\calW}_Q^\cpp W(\cdot)\ve_i|
   \chi_Q(\cdot)\|_{\Lpp(\Omega)}\\
	& \qquad	\approx |Q|^{-1/p_Q}\|\sum_{i=1}^d
   |\overline{\calW}_Q^\cpp W(\cdot)\ve_i|
   \chi_Q(\cdot)\|_{\Lpp(\Omega)}\\
	&\qquad		\approx |Q|^{-1/p_Q}
   \||\overline{\calW}_Q^\cpp W(\cdot)|_{\op}
   \chi_Q(\cdot)\|_{\Lpp(\Omega)}.
\end{align*}

Since $\overline{\calW}_Q^\cpp$ and $W$ are selfadjoint, we can
commute them within the operator norm. Then, by Lemma
\ref{opNorm:equiv}, inequality \eqref{SumNorm:equiv}, and the
definition of the reducing operator $\calW_Q^\pp$, we get
\begin{align*}
&|Q|^{-1/p_Q}\||W(\cdot)\overline{\calW}_Q^\pp|_{\op}\chi_Q(\cdot)\|_{\Lpp(\Omega)}\\
	 & \qquad\approx |Q|^{-1/p_Q} \left\|\sum_{i=1}^d |W(\cdot)\overline{\calW}_Q^\cpp\ve_i|\chi_Q(\cdot) \right\|_{\Lpp(\Omega)}\\
	 & \qquad \approx \sum_{i=1}^d |Q|^{-1/p_Q} \||W(\cdot)\overline{\calW}_Q^\cpp \ve_i|\chi_Q(\cdot)\|_{\Lpp(\Omega)}\\
	 & \qquad \approx \sum_{i=1}^d |\calW_Q^\pp \overline{\calW}_Q^\cpp \ve_i|\\
	  & \qquad\approx |\calW_Q^\pp \overline{\calW}_Q^\cpp|_{\op}.
\end{align*}
The implicit constants do not depend on $Q$, and so we have \eqref{ApdotReducingOps:ineq}.
\end{proof}
\begin{proof}[Proof of Theorem~\ref{AvgOpBddEquivApp}]
We first prove the forward implication. Let $W \in \calA_\pp$ and $Q\subseteq \Omega$ be a cube. Fix a function $\vf\in \Lpp(W,\Omega)$. By Lemma \ref{Holder} and the definition of $\calA_\pp$, 
\begin{align*}
  \|A_Q \vf\|_{\Lpp(W,\Omega)}
  & = \bigg\| W(x) \dashint_Q \vf(y)\, dy \;\chi_Q(x)\bigg\|_{\Lpp_x(\Omega)}\\
	& \leq |Q|^{-1}\bigg\|  \int_Q |W(x)W^{-1}(y) W(y) \vf(y)| \,
   dy
   \;\chi_Q(x) \bigg\|_{\Lpp_x(\Omega)}\\
	& \leq |Q|^{-1} \bigg\| \int_Q |W(x)W^{-1}(y)|_{\op} |
   W(y)\vf(y)| \,dy
   \; \chi_Q(x)\bigg\|_{\Lpp_x(\Omega)}\\
	& \leq 4|Q|^{-1} \big\|  \||W(x)W^{-1}(y)|_{\op}
   \chi_Q(y)\|_{\Lcpp_y(\Omega)}
   \|\vf\|_{\Lpp(W,\Omega)} \;\chi_Q(x)\big\|_{\Lpp_x(\Omega)}\\
	&\leq 4[W]_{A_\pp} \| \vf\|_{\Lpp(W,\Omega)}.
\end{align*}
Thus, $A_Q : \Lpp(W,\Omega) \to \Lpp(W,\Omega)$ uniformly for all cubes $Q \subseteq \Omega$.

To prove the converse, we will show that $[W]_{\calA_\pp}^R
<\infty$. Fix a cube  $Q\subseteq \Omega$ and fix $\vv \in \R^d$
with $|\vv|=1$. By the equivalence between the reducing operator
$\overline{\calW}_Q^\cpp$ and the $\Lcpp(W^{-1},\Omega)$ average,
Proposition \ref{VectorDuality}, the self-adjointness of $W^{-1}$ and
$\calW_Q^\pp$, and the linearity of the dot product, we have that for some
$\vg \in \Lpp(\Omega;\R^d)$ with $\|\vg\|_{\Lpp(\Omega;\R^d)}\leq d$,
\begin{align*}
|\overline{\calW}_Q^\cpp \calW_Q^\pp \vv| &\leq \sqrt{d} |Q|^{-1/p'_Q} \||W^{-1}(\cdot)\calW_Q^\pp \vv|\chi_Q(\cdot)\|_{\Lcpp(\Omega)}\\
	& \leq 4\sqrt{d}|Q|^{-1/p'_Q} \int_Q W^{-1}(y) \calW_Q^\pp \vv \cdot \vg(y) \,dy\\
	& = 4\sqrt{d}|Q|^{-1/p'_Q}  \int_Q \vv \cdot\calW_Q^\pp W^{-1}(y) \vg(y) \,dy\\
	& = 4\sqrt{d}|Q|^{-1/p'_Q} \vv \cdot \calW_Q^\pp \int_Q W^{-1}(y) \vg(y)\, dy\\
	& \leq 4\sqrt{d}|Q|^{-1/p'_Q} \left| \calW_Q^\pp \int_Q W^{-1}(y) \vg(y) \,dy\right|.
\end{align*}
Now, using the equivalence between the reducing operator $\calW_Q^\pp$
and the $\Lpp(W,\Omega)$ average, we have that the last expression is
bounded by
\begin{multline*}
4d|Q|^{-1}\left\| W(x) \int_Q W^{-1}(y)\vg(y) dy\;
  \chi_Q(x)\right\|_{\Lpp_x(\Omega)} \\
= 4d\|W(x)A_Q(W^{-1}\vg)(x)\|_{\Lpp_x(\Omega)} 
	 = 4d \|A_Q(W^{-1} \vg)\|_{\Lpp(W,\Omega)}.
\end{multline*}

By assumption, $A_Q: \Lpp(W,\Omega) \to \Lpp(W,\Omega)$ uniformly over all cubes $Q\subseteq \Omega$. Thus,
\[\|A_Q(W^{-1}\vg)\|_{\Lpp(W,\Omega)} \leq \| A_Q\| \|W^{-1} \vg\|_{\Lpp(W,\Omega)} = \|A_Q\| \|\vg\|_{\Lpp(\Omega;\R^d)}.\]
Since $\|\vg\|_{\Lpp(\Omega;\R^d)} \leq d$, we have that $[W]_{\calA_\pp}^R < 4d^2
\|A_Q\| <\infty$. Hence,  by Proposition~\ref{ApdotReducingOps}, $W\in \calA_\pp$.
\end{proof}
\begin{proof}[Proof of Theorem~\ref{SumAvgOpBdd}]
  Fix a collection $\calQ$ of pairwise disjoint cubes contained in
  $\Omega$, and fix $\vf \in \Lpp(W,\Omega)$.  Extend $\vf $ to equal
  $0$ on $\R^n \bk \Omega$. By
  Proposition~\ref{VectorDuality} applied to $WA_\calQ \vf$, there
  exists a function $\vh \in \Lcpp(\Omega;\R^d)$ with
  $\| \vh\|_{\Lcpp(\Omega;\R^d)}\leq d$ such that
\[\| WA_{\calQ} \vf\|_{\Lpp(\Omega;\R^d)} \leq 4 \int_{\Omega} W(x) A_\calQ \vf(x) \cdot \vh(x) \, dx.\] 
Let $\vg = W\vh$. Then, $W^{-1}\vg\in \Lcpp(\Omega;\R^d)$ with
$\|W^{-1}\vg\|_{\Lcpp(\Omega;\R^d)}\leq d$. Thus,
\begin{align*}
\|WA_\calQ \vf\|_{\Lpp(\Omega;\R^d)}&\leq  4\int_{\Omega} W(x) A_\calQ \vf(x)\cdot W^{-1}(x)\vg(x) \, dx \\
 	& \leq 4\sum_{Q\in \calQ} \int_Q \dashint_Q |W(x) W^{-1}(y)W(y)\vf(y)| \, dy\, |W^{-1}(x) \vg(x)| \, dx\\
	& \leq 4 \sum_{Q\in \calQ} \int_Q \dashint_Q|W(x)W^{-1}(y)|_{\op}|W(y)\vf(y)| \, dy \, W^{-1}(x) \vg(x) \, dx.
\end{align*}
We next use Lemma \ref{Holder} twice, and then, since
$\pp \in LH(\Omega)$, by Lemma \ref{LH-G:lemma} we can use property
$\calG$ to bound the previous expression by
\begin{align*}
 & 16\sum_{Q\in \calQ} \int_{Q}|Q|^{-1}\||W(x)W^{-1}(y)|_{\op}\chi_Q(y)\|_{\Lcpp_y(\Omega)}\|\vf \chi_Q\|_{\Lpp(W,\Omega)} W^{-1}(x)\vg(x) \, dx\\
	& \qquad \leq 64[W]_{\calA_\pp} \sum_{Q\in\calQ}\|\vf\chi_Q\|_{\Lpp(W,\Omega)} \|\vg\chi_Q\|_{\Lcpp(W^{-1},\Omega)}  \\
	& \qquad\leq 64 [W]_{\calA_\pp} C(n,\pp)\|\vf\|_{\Lpp(W,\Omega)}\|\vg\|_{\Lcpp(W^{-1},\Omega)}\\
	& \qquad\leq 64[W]_{\calA_\pp}C(n,d,\pp)\|\vf\|_{\Lpp(W,\Omega)}.
\end{align*}
This completes the forward implication.

To prove the converse, observe that if \eqref{SumAvgOpBdd:ineq} holds for any collection $\calQ$ of pairwise disjoint cubes $Q\subseteq \Omega$, then it holds for any individual cube $Q\subseteq \Omega$. Thus, $A_Q: \Lpp(W,\Omega)\to \Lpp(W,\Omega)$ uniformly for all cubes $Q\subseteq \Omega$, and so by Theorem \ref{AvgOpBddEquivApp}, $W \in \calA_\pp$.
\end{proof}

\medskip

Finally, we prove that the operator norm of a matrix in $\A_\pp$
is a scalar weight in $A_\pp$.  In the constant exponent case, this
was proved in~\cite[Section 4]{cruz-uribe_matrix_2016}.

\begin{prop}\label{opNormApdot}
Given a matrix weight $W: \Omega \to \calS_d$, let $v(x)=|W(x)|_{\op}$. If $W\in \calA_\pp$, then $v \in A_\pp$.
\end{prop}

To prove this, we need two preliminary results.  First, we need the
scalar version of Theorem~\ref{AvgOpBddEquivApp}.

\begin{prop}\cite[Proposition 3.7]{Cruz-Uribe-FiorenzaNeugebauer-2012}\label{AvgOpBddEquivApp-Scalar}
Let $\pp \in \Pp(\Omega)$ and $w: \Omega \to [0,\infty)$ be a weight. Then $w\in A_\pp$ if and only if $A_Q : \Lpp(w,\Omega) \to \Lpp(w,\Omega)$ uniformly for all cubes $Q\subseteq \Omega$. Moreover,
\[ \|A_Q\| \leq  4[w]_{A_\pp} \leq C \|A_Q\|.\]
\end{prop}

Second, we need a technical lemma on $A_\pp$ weights. 

\begin{lemma}\label{AppAddition}
Given $\pp \in \Pp(\Omega)$, if $w_1, \ldots, w_d\in A_\pp$, then, $\sum_{j=1}^d w_j \in A_\pp$.
\end{lemma}

\begin{proof}
  Fix $w_1, \ldots, w_d\in A_\pp$. Since weights are positive almost
  everywhere, $\left( \sum_{j=1}^d w_j \right)^{-1} \leq w_i^{-1}$ for
  each $i=1,\ldots, d$. If we comine this with the triangle inequality
  and the definition of $A_\pp$, we have for all cubes
  $Q\subseteq \Omega$,
\begin{align*}
&  |Q|^{-1}\bigg\|\chi_Q\bigg(\sum_{j=1}^d
                 w_j\bigg)^{-1}\bigg\|_{\Lcpp(\Omega)}
                 \bigg\| \chi_Q\bigg( \sum_{i=1}^d w_i\bigg)\bigg\|_{\Lpp(\Omega)}\\
  & \qquad \leq  \sum_{i=1}^d |Q|^{-1}
    \bigg\| \chi_Q \bigg( \sum_{j=1}^d w_j\bigg)^{-1}
    \bigg\|_{\Lcpp(\Omega)}
    \| \chi_Q w_i\|_{\Lpp(\Omega)}\\
  & \qquad  \leq  \sum_{i=1}^d |Q|^{-1}
    \| \chi_Q w_i^{-1}\|_{\Lcpp(\Omega)} \| \chi_Q w_i \|_{\Lpp(\Omega)}\\
	& \qquad  \leq \sum_{i=1}^d [w_i]_{A_\pp}\\
	& \qquad <\infty.
\end{align*}

 Hence, $\sum_{i=1}^d w_i \in A_\pp$.
\end{proof}

\begin{proof}[Proof of Proposition~\ref{opNormApdot}]
  Let $\{\ve_i\}_{i=1}^d$ be the standard coordinate basis in
  $\R^n$. By Proposition \ref{opNorm:equiv},
  $v(x)\approx \sum_{i=1}^d |W(x)\ve_i|$, so by Lemma
  \ref{AppAddition} it will suffice to show that
  $|W(x)\ve_i|\in A_\pp$ for each $i=1,\ldots, d$. Fix $i$, let
  $Q\subseteq \Omega$ be a cube, and let
  $f\in \Lpp(|W \ve_i|,\Omega)$. Then,
\begin{align*}
\| |W(\cdot)\ve_i| A_Q f(\cdot)\|_{\Lpp(\Omega)} & = \left\| \dashint_Q |W(x)\ve_i| f(y) \, dy \;\chi_Q(x)\right\|_{\Lpp_x(\Omega)}\\
	& \leq \left\| \dashint_Q |W(x) W^{-1}(y)|_{\op} |W(y)\ve_i|f(y) \, dy \; \chi_Q(x)\right\|_{\Lpp_x(\Omega)}\\
	& \leq  |Q|^{-1}\| \||W(x)W^{-1}(y)|_{\op}\;\chi_Q(y)\|_{\Lcpp_y(\Omega)}\chi_Q(x)\|_{\Lpp_x} \||W\ve_i|f\|_{\Lpp(\Omega)} \\
& \leq [W]_{\calA_\pp} \| |W\ve_i|f\|_{\Lpp(\Omega)}.
\end{align*}
Thus, $A_Q: \Lpp(|W\ve_i|,\Omega) \to \Lpp(|W\ve_i|,\Omega)$ uniformly for any cube $Q\subseteq \Omega$, so by Proposition \ref{AvgOpBddEquivApp-Scalar} we have that $|W\ve_i|\in A_\pp$. 
\end{proof}
%

%
%
%
\section{Convergence Using Approximate Identities}\label{ConvConvergence:sec}

In this section we prove
Theorem~\ref{ConvolutionConvergence:theorem}.  The heart of the proof
is the following lemma.  Convolution with $|Q|^{-1}\chi_Q$ acts
like an averaging operator, and the proof of this result is very
similar to the proof of Theorem~\ref{SumAvgOpBdd}.

\begin{lemma}\label{ConvLocallyBdd}
Fix $\pp\in LH(\Omega)$ and $W \in \calA_\pp$. Given any cube $Q$ centered at the origin and $\vf\in \Lpp(W,\Omega)$, 
\[\| |Q|^{-1} \chi_Q \ast \vf\|_{\Lpp(W,\Omega)} \leq C(n,d,\pp) [W]_{\calA_\pp} \|\vf\|_{\Lpp(W,\Omega)}.\]
The same inequality is true if we replace the cube $Q$ with any ball $B$ centered at the origin.
\end{lemma}

\begin{remark}
  We want to emphasize that for this proof to work, we need that the
  cubes or balls are centered at the origin.  In the proof of the
  constant exponent version of this result
  in~\cite[Lemma~4.8]{cruz-uribe_matrix_2016}, this fact was used implicitly, but
  was not made clear in the statement of the result.
\end{remark}
\begin{proof}
  Fix a cube $Q$ centered at the origin. For each $\vk\in \Z^n$,
  define $Q_\vk = Q+\ell(Q)\vk$. Then $\{Q_\vk\}_{\vk\in \Z^n}$ forms
  a partition of $\R^n$. For a fixed $\vk$, let $x \in Q_\vk$. Then,
  since $Q$ is centered at the origin,
  $\{y \in \R^n: x-y \in Q\} \subseteq 3Q_\vk$. Let
  $\vf\in \Lpp(W,\Omega)$ and extend $\vf(x)$ to equal $0$ on
  $\R^n\backslash \Omega$. 
  Then, for all $\vk\in \Z^n$,

\begin{align*}
  | W(x) |Q|^{-1}(\chi_Q\ast \vf)(x)|
  & =  \left| W(x) |Q|^{-1}\int_{\R^n} \vf(y) \chi_Q(x-y)\, dy\right|\\
	& \leq |Q|^{-1}\int_{\R^n}|W(x)\vf(y)|\chi_Q(x-y)\, dy\\
	& \leq C(n) \dashint_{3Q_{\vk}} |W(x) \vf(y)|\,dy.\\	
\end{align*}

Since $\pp \in LH(\Omega)$, by Lemma
  \ref{Extension:lemma}, we can extend $\pp$ to be in $LH(\R^n)$. By Proposition \ref{VectorDuality}, there exists $\vg \in \Lcpp(\R^n)$ with $\| \vg\|_{\Lcpp(\R^n)} \leq d$ such that
\begin{align*}
& \| |Q|^{-1}\chi_Q \ast \vf\|_{\Lpp(W,\Omega)} \\
	& \qquad = \left\| \sum_{\vk \in \Z^n} (|Q|^{-1}\chi_Q \ast \vf) \chi_{Q_\vk}\right\|_{\Lpp(W,\Omega)}\\
	& \qquad = \left\| \sum_{\vk\in\Z^n} W(|Q|^{-1}\chi_Q\ast \vf ) \chi_{Q_\vk} \right\|_{\Lpp(\Omega)}\\
	& \qquad \leq 4 \int_{\R^n} \left( \sum_{\vk\in\Z^n} W(x) ( |Q|^{-1}\chi_Q \ast \vf)(x)\chi_{Q_\vk}(x)\right) \cdot \vg(x) \,dx\\
	& \qquad = 4\sum_{\vk\in\Z^n}\int_{Q_{\vk}} W(x) ( |Q|^{-1}\chi_Q \ast \vf)(x) \cdot \vg(x) \,dx\\
	& \qquad \leq 4 \sum_{\vk\in\Z^n} \int_{Q_\vk} C(n) \dashint_{3Q_\vk} |W(x) \vf(y)  | \,dy \,|\vg(x)|\, dx\\
	& \qquad \leq 4 C(n) \sum_{\vk\in \Z^n} \int_{Q_\vk} |3Q_\vk|^{-1} \int_{3Q_\vk} |W(x)W^{-1}(y)|_{\op} |W(y) \vf(y)| \, dy \, |\vg(x)| \, dx.
\end{align*}

Using Lemma \ref{Holder} twice, we bound the previous expression by
\begin{align*}
  &16C(n) \sum_{\vk\in\Z^n} \int_{Q_\vk} |3Q_\vk|^{-1}
    \| |W(x)W^{-1}(\cdot)|_{\op} \chi_{3Q_\vk}(\cdot) \|_{\Lcpp(\R^n)}
    \| \vf\chi_{3Q_\vk}\|_{\Lpp(W,\R^n)} |\vg(x)| \, dx\\
  & \quad \leq 64 C(n) [W]_{\calA_\pp} \sum_{\vk\in \Z^n}
    \| \vf\chi_{3Q_\vk}\|_{\Lpp(W,\R^n)} \| \vg\chi_{3Q_\vk}\|_{\Lcpp(\R^n)}.
\end{align*}

We can divide the cubes $\{3Q_\vk\}_{\vk \in \Z^n}$ into $3^n$ families $\calQ_i$ of pairwise disjoint cubes. Since $\pp \in LH(\R^n)$, by Lemma \ref{LH-G:lemma}, $\pp \in \calG$. Since each family $\calQ_i$ is pairwise disjoint, $\|\vg\|_{\Lpp(\R^n)}\leq d$, and $\vf=0$ outside $\Omega$, we have
\begin{align*}
\sum_{\vk \in \Z^n} \| \vf \chi_{3Q_\vk}\|_{\Lpp(W,\R^n)} \| \vg \chi_{3Q_\vk}\|_{\Lcpp(\R^n)} & \leq \sum_{i=1}^{3^n} \sum_{P\in \calQ_i} \| \vf \chi_P \|_{\Lpp(W,\R^n)} \| \vg \chi_P\|_{\Lcpp(\R^n)}\\
    & \leq 3^n C(n,\pp) \| \vf\|_{\Lpp(W,\R^n)}\| \vg\|_{\Lcpp(\R^n)}\\
    & \leq 3^n C(n,\pp) d \| \vf\|_{\Lpp(W,\Omega)}.
\end{align*}
This completes the proof for cubes.

To prove this result for balls, fix a ball $B$ centered at the origin
and let $Q$ be the smallest cube centered at the origin containing
$B$. Then $|B| \approx |Q|$, and so for any $x\in \R^n$,
\begin{align*}
|W(x) |B|^{-1} (\chi_B \ast \vf)(x)| & = \left|W(x)|B|^{-1} \int_{\R^n} \vf(y) \chi_{B}(x-y) \, dy \right|\\
    & \leq |B|^{-1} \int_{\R^n} |W(x) \vf(y)| \chi_B(x-y) \, dy\\
    & \leq C(n) |Q|^{-1}\int_{\R^n} |W(x) \vf(y)| \chi_Q(x-y) \, dy.\\
\end{align*}
The proof then continues as before.  
\end{proof}

We now prove our main result about the convergence of convolution operators.
\begin{proof}[Proof of Theorem \ref{ConvolutionConvergence:theorem}]
Fix a function $\vf \in \Lpp(W,\Omega)$ and extend $\vf$ to $0$ on
$\R^n \bk \Omega$.  Define  the function 
\[\Phi(x)=\sum_{k=1}^\infty a_k|B_k|^{-1} \chi_{B_k}(x),\]
where $\{B_k\}_{k=1}^\infty$ is a sequence of balls, each centered at
the origin with $B_{k+1}\subset B_k$ for all $k$, and the $a_k$ are
nonnegative with $\sum_{k=1}^\infty a_k =1$.  To prove
inequality~\eqref{ConvolutionBound} it will suffice to prove that
\begin{equation} \label{eqn:Phi-bound}
\|\Phi\ast \vf\|_{\Lpp(W,\Omega)} \leq C(n,\pp)[W]_{\calA_\pp}
\|\vf\|_{\Lpp(W,\Omega)};
\end{equation}
for if this holds, then inequality \eqref{ConvolutionBound} will
follow if we approximate $\phi_t$ from below by a sequence of such
functions and apply Fatou's lemma for variable Lebesgue spaces (see
\cite[Theorem 2.61]{cruz-uribe_variable_2013}).

But to prove inequality~\eqref{eqn:Phi-bound}, note that by Minkowski's
inequality and Lemma \ref{ConvLocallyBdd}, we have
\begin{align*}
  \|\Phi\ast \vf\|_{\Lpp(W,\Omega)}
  & \leq \sum_{k=1}^\infty a_k \||B_k|^{-1} \chi_{B_k}\ast \vf\|_{\Lpp(W,\Omega)}\\
	& \leq C(n, \pp) [W]_{\calA_\pp} \sum_{k=1}^\infty a_k \|\vf\|_{\Lpp(W,\Omega)}\\
	& \leq C(n,\pp) [W]_{\calA_\pp} \|\vf\|_{\Lpp(W,\Omega)}.
\end{align*}

We now prove the limit~\eqref{ConvolutionConvergence:ineq}. Fix
$\epsilon>0$. Since $p_+<\infty$, by Lemma
\ref{ContCompSuppDense}, $C_c^\infty(\Omega;\R^d)$ is dense in
$\Lpp(W,\Omega)$, and so there exists
$\vg \in C_c^\infty(\Omega;\R^d)$ such that
$\|\vf-\vg\|_{\Lpp(W,\Omega)}<\epsilon$.

By a classical result (see~\cite[Theorem~8.14]{MR1681462}),
$\phi_t\ast \vg \to \vg$ uniformly as $t\to 0$.
Since $W \in \calA_\pp$, $|W|_{\op}\in A_\pp$ by Proposition \ref{opNormApdot}. Consequently, $|W|_{\op} \in \Lpp(\Omega)$ by the definition of $A_\pp$. Thus, for $t$ sufficiently small,
\[ \|\phi_t\ast \vg - \vg\|_{\Lpp(W,\Omega)} \leq \| |W|_{\op} |\phi_t\ast \vg - \vg|\|_{\Lpp(\Omega)} \leq \epsilon.\]

Combining this with Lemma \ref{ConvLocallyBdd}, we have that for $t$ sufficiently small,
\begin{align*}
\|\phi_t \ast \vf - \vf\|_{\Lpp(W,\Omega)} & \leq \|\phi_t\ast \vf-\phi_t\ast \vg\|_{\Lpp(W,\Omega)} + \|\phi_t\ast \vg - \vg\|_{\Lpp(W,\Omega)} + \|\vg - \vf\|_{\Lpp(W,\Omega)}\\
	& \leq (1+C(n,\pp)[W]_{\calA_\pp})\|\vf-\vg\|_{\Lpp(W,\Omega)} + \|\phi_t\ast \vg - \vg\|_{\Lpp(W,\Omega)}\\
	& \lesssim \epsilon.
\end{align*}
Since $\epsilon>0$ was arbitrary, \eqref{ConvolutionConvergence:ineq}
follows at once.
\end{proof}


%
%
%
\section{Application: $\mathscr{H}=\mathscr{W}$}\label{H=W:sec}

In this section, we prove Theorems~\ref{H=W:theorem}
and~\ref{H=W-scalar:theorem} and Corollary~\ref{cor:smooth-dense-Rn}.
We will concentrate on the results for the Sobolev space $\scrWpp(W,\Omega)$;
the analogous results for  $\scrWpp(v,W,\Omega)$ are proved similarly
and we will briefly sketch the changes.   

We first recall a few definitions on weak derivatives;  see~\cite{gilbarg_elliptic_2001}
for more information.  A function $f\in L^1_{\text{loc}}(\Omega)$ is weakly
differentiable, or has derivatives in the distributional sense, with respect to $x_j$, if
there exists a function $g_j \in L^1_{\text{loc}}(\Omega)$ such that
for every $\phi \in C_c^\infty(\Omega)$,
\[ \int_\Omega g_j(x)\phi(x)\,dx
  = - \int_\Omega f(x) \partial_j \phi(x)\,dx.  \]
We will denote this function $g_j$ by $\partial_j f$. Define
$\scrWloc(\Omega)$ to be the collection of functions $f\in
L^1_{\text{loc}}(\Omega)$ such that $\partial_j f$ exists for
$j=1,\ldots, n$.  
If $\vf = (f_1, \ldots, f_d)^T \in L^1_{\text{loc}}(\Omega;\R^d)$,
then define the weak derivative of $\vf$ with respect to $x_j$ by
$\partial_j \vf = (\partial_j f_1, \ldots, \partial_j f_d)^T$. Define
$\scrWloc(\Omega;\R^d)$ to be the set of functions
$\vf \in L^1_{\text{loc}}(\Omega;\R^d)$ such that $\partial_j f_i$
exists for $j=1,\ldots, n$, $i=1,\ldots, d$.  Let $D\vf
=(\partial_j f_i)_{i,j}$ be the Jacobian matrix of $\vf$. 

Given $\pp \in \Pp(\Omega)$ and a matrix weight
$W : \Omega \rightarrow \calS_d$, define $\scrWpp(W,\Omega)$ to be the
collection of $\vf \in \scrWloc(\Omega;\R^d)$ such that
\begin{align}\label{SobSpaceNorm}
\|\vf\|_{\scrWpp(W,\Omega)} := \|\vf\|_{\Lpp(W,\Omega)} + \| |W D\vf|_{op}\|_{\Lpp(\Omega)} <\infty.
\end{align}
Similarly, for an $n\times n$ matrix weight $W:\Omega \to \calS_n$, let $v=|W|_\op$ and define $\scrWpp(v,W,\Omega)$ to be
the collection of $f\in \scrWloc(\Omega)$ such that
\[ \|f\|_{\scrWpp(v,W,\Omega)} = \|f\|_{\Lpp(v,\Omega)} + \| \nabla
  f\|_{\Lpp(W,\Omega)} < \infty. \]

Our definition of $\scrWpp(W,\Omega)$
follows~\cite{isralowitz_matrix_2019} for constant exponents.
In~\cite{MR2104276} the authors gave an equivalent definition of this
space, which  they
denoted $L^p_1(W,\Omega)$.  Following them,  define the Sobolev
space $\Lpp_1(W,\Omega)$ to be the collection of functions
$\vf : \Omega \to \R^d$ such that
\begin{align}\label{SobSpaceNorm:Frazier}
\| \vf\|_{\Lpp_1(W,\Omega)} : = \| \vf\|_{\Lpp(W,\Omega)} + \sum_{j=1}^n \| \partial_j \vf\|_{\Lpp(W,\Omega)}.
\end{align}
As in the constant exponent case, the  spaces $\scrWpp(W,\Omega)$ and $\Lpp_1(W,\Omega)$ are equivalent.

\begin{prop}\label{W=L}
Let $\pp \in \Pp(\Omega)$ and $W : \Omega \to \calS_d$ be a matrix weight. Then $\scrWpp(W,\Omega) = \Lpp_1(W,\Omega)$ with equivalent norms.
\end{prop}
\begin{proof}
  Let $\vf=(f_1, \ldots, f_d)^T \in \scrWloc(\Omega)$.  Let
  $\{\ve_j\}_{j=1}^n$ be the standard coordinate basis in $\R^n$. Then
  by Lemma \ref{opNorm:equiv},
\[ \| |WD\vf|_{op}\|_{\Lpp(\Omega)} \approx \left\| \sum_{j=1}^n |W D\vf \ve_j|\right\|_{\Lpp(\Omega)} = \left\| \sum_{j=1}^n |W\partial_j \vf|\right\|_{\Lpp(\Omega)}.\]
By inequality \eqref{SumNorm:equiv}, 
\[ \frac{1}{n}\sum_{j=1}^n  \| |W\partial_j \vf|\|_{\Lpp(\Omega)} \leq \left\| \sum_{j=1}^n |W\partial_j \vf|\right\|_{\Lpp(\Omega)} \leq  \sum_{j=1}^n  \| |W\partial_j \vf|\|_{\Lpp(\Omega)} .\]
Hence, $\| \vf\|_{\scrWpp(W,\Omega)} \approx \| \vf\|_{\Lpp_1(W,\Omega)}$ with implicit constants depending only on $n$.
\end{proof}

\begin{theorem}\label{WBanach}
Let $\pp \in \Pp(\Omega)$ and $W \in \calA_\pp$. Then
$\scrWpp(W,\Omega)$ is a Banach space.  Similarly,
$\scrWpp(v,W,\Omega)$ is a Banach space.
\end{theorem}
\begin{proof}
We will prove that $\Lpp_1(W,\Omega)$ is a Banach space. Proposition \ref{W=L} then implies $\scrWpp(W,\Omega)$ is also a Banach space; the proof
for $\scrWpp(v,W,\Omega)$ is essentially the same (using the fact that
$\Lpp(v,\Omega)$ is a Banach space:
see~\cite[Theorem~2.5]{Penrod-Poincare-2022}) and so is omitted.  Our
proof is adapted from the proof
of~\cite[Theorem~5.2]{cruz-uribe_matrix_2016}.

Since we can map each $\vf\in \Lpp_1(W,\Omega)$ to a unique
$(n+1)$-tuple $(\vf, \partial_1 \vf, \ldots, \partial_n \vf)$, we
may consider $\Lpp_1(W,\Omega)$ as a linear subspace of
$\bigoplus_{j=0}^n \Lpp(W,\Omega)$. Since $\Lpp(W,\Omega)$ is a Banach
space (see \cite[Theorem 2.12]{Penrod-Poincare-2022}), so is
$\bigoplus_{j=0}^n \Lpp(W,\Omega)$. Therefore, to complete the proof it will suffice  to show
$\Lpp_1(W,\Omega)$ is a closed subspace of
$\bigoplus_{j=0}^n \Lpp(W,\Omega)$.

Fix a Cauchy sequence $\{\vu_k\}_{k=1}^\infty$ in
$\Lpp_1(W,\Omega)$. Since $\Lpp(W,\Omega)$ is
complete, there exist $\vv,\vv_1, \ldots, \vv_n\in \Lpp(W,\Omega)$
such that $\vu_k \to \vv$ and $\partial_j(\vu_k) \to \vv_j$ for
$j=1,\ldots, n$ in $\Lpp(W,\Omega)$. For each $j=1,\ldots, n$, we will show that
$\partial_j \vv= \vv_j$, i.e., for all $\phi \in C_c^\infty(\Omega)$,
\begin{align}\label{WBanach:eq}
\int_\Omega \vv_j(x) \phi(x)\, dx = - \int_{\Omega} \vv(x) \partial_j\phi(x) \,dx.
\end{align}

Fix $\phi \in C_c^\infty(\Omega)$ and let $K=\supp(\phi)$. We first
show both integrals are finite. For the first, observe that by Lemma~\ref{Holder},
\begin{multline*}
  \int_{\Omega} |\vv_j(x) \phi(x)| \,dx
  \leq \| \phi\|_{L^\infty(\Omega)} \int_K |W^{-1}(x) W (x)\vv_j(x)|
  \, dx \\
\leq 4\| \phi\|_{L^\infty(\Omega)} \| |W^{-1}|_{\op} \chi_K\|_{\Lcpp(\Omega)}\| \vv_j\|_{\Lpp(W,\Omega)}.
\end{multline*}
Since $W \in \calA_\pp$, $W^{-1}\in \calA_\cpp$. Thus, by Proposition
\ref{opNormApdot}, $|W^{-1}|_{\op} \in A_\cpp$. Hence,
$\| |W^{-1}|_{\op} \chi_K\|_{\Lcpp(\Omega)} <\infty$. Furthermore,
$\vv_j \in \Lpp(W,\Omega)$. Hence, the first integral in
\eqref{WBanach:eq} is finite.
Similarly, to see the second integral in \eqref{WBanach:eq} is finite, observe that
\begin{multline*}
  \int_\Omega |\vv(x) \partial_j \phi(x)| \, dx
    \leq \| \partial_j \phi\|_{L^\infty(\Omega)} \int_K |W^{-1}(x)W (x)\vv(x)| \, dx\\
    \leq 4\| \partial_j \phi \|_{L^\infty(\Omega)} \| |W^{-1}|_{\op} \chi_K \|_{\Lcpp(\Omega)} \| \vv\|_{\Lpp(W,\Omega)} <\infty.
\end{multline*}

We now show \eqref{WBanach:eq} holds.  If we integrate by parts, we
have have that
\begin{align*}
  \left| \int_{\Omega} \vv_j \phi + \vv \partial_j \phi \, dx \right|
  & \leq \left| \int_\Omega (\vv_j - \partial_j(\vu_k)) \phi \, dx \right| + \left| \int_{\Omega} (\vv - \vu_k) \partial_j \phi \, dx\right|\\
    &  \leq 4\| \phi \|_{L^\infty(\Omega)} \||W^{-1}|_{\op}\chi_K\|_{\Lcpp(\Omega)} \| \vv_j - \partial_j(\vu_k)\|_{\Lpp(W,\Omega)} \\
 & \qquad \qquad     +4 \| \partial_j \phi\|_{L^\infty(\Omega)} \| |W^{-1}|_{\op} \chi_K \|_{\Lcpp(\Omega)}\| \vv - \vu_k\|_{\Lpp(W,\Omega)}.
\end{align*}

Both norms on the right go to zero as $k \to \infty$, and so $\vv_j = \partial_j \vv$ for $j=1,\ldots, n$. Hence, $\vv \in \Lpp_1(W,\Omega)$. Thus, $\Lpp_1(W,\Omega)$ is a closed
subspace of the Banach space $\bigoplus_{j=0}^n \Lpp(W,\Omega)$, and
hence is a Banach space.
\end{proof}

\begin{definition}\label{SobSpaceH:def}
Let $\pp \in \Pp(\Omega)$. For a $d\times d$ matrix weight $W \in \A_\pp$, define $\scrHpp(W,\Omega)$ to be the closure of
$C^\infty(\Omega;\R^d)\cap \scrWpp(W,\Omega)$ in  $\scrWpp(W,\Omega)$.
Similarly,
for a $n\times n$ matrix weight $W \in \A_\pp$, define $\scrHpp(v,W,\Omega)$ to be the closure of
$C^\infty(\Omega)\cap \scrWpp(v, W,\Omega)$ in $\scrWpp(v,
W,\Omega)$. 
\end{definition}

We can now prove Theorems~\ref{H=W:theorem}
and~\ref{H=W-scalar:theorem}.  
\begin{proof}[Proof of Theorem \ref{H=W:theorem}]
  The proof is an adaptation of the proof that
  $\mathscr{H}^{1,p}(W,\Omega)=\mathscr{W}^{1,p}(W,\Omega)$ with
  matrix weights in the constant exponent setting:
  see~\cite[Theorem~5.3]{cruz-uribe_matrix_2016}. It is immediate from
  the definition that $\scrHpp(W,\Omega) \subseteq \scrWpp(W,\Omega)$.
  We will show that given any $\vf\in \scrWpp(W,\Omega)$ and any
  $\epsilon>0$, there exists
  $\vg\in C^\infty(\Omega;\R^n) \cap \scrWpp(W,\Omega)$ such that
  $\|\vf-\vg\|_{\scrWpp(W,\Omega)} < \epsilon$.

For each $k \in \N$, define the bounded sets 
\[\Omega_k = \{x\in \Omega : |x|<k, \dist(x,\partial \Omega) >1/k\}.\]
Let $\Omega_0= \Omega_{-1} = \emptyset$, and define the sets
$A_k = \Omega_{k+1}\bk \overline{\Omega}_{k-1}$. These sets are an
open cover of $\Omega$, each $\overline{A}_k$ is compact, and given
$x\in \Omega$, $x\in A_k$ for a finite number of indices $k$. Thus, we
can form a partition of unity subordinate to this cover
(see~\cite[Proposition~4.41]{MR1681462}): there exists
$\psi_k\in C_c^\infty(A_k)$ such that for all $x\in \Omega$,
$0\leq \psi_k(x)\leq 1$ and
\[\sum_{k=1}^\infty \psi_k(x)=1.\]

Fix $\vf\in \scrWpp(W,\Omega)$. Then $\vf\in \scrWloc(\Omega)$, and so $\psi_j \vf\in \scrWloc(\Omega)$. Fix a nonnegative, radially symmetric and decreasing function $\phi \in C_c^\infty(B(0,1))$ with $\int_{B(0,1)} \phi(x) \, dx = 1$. For all $t>0$, define $\phi_t(x) = t^{-n}\phi(x/t)$. Then the convolution
\[\phi_t \ast (\psi_k \vf) (x) = \int_{A_k} \phi_t(x-y) \psi_k(y)\vf(y) \, dy\]
is only non-zero if for some $y\in A_k$, $|x-y|<t$. Hence, for $k\geq 3$, we can choose $t=t_k$ such that $0 <t_k < (k+1)^{-1} -(k+2)^{-1}$. Then, there exists $y\in A_k$ with $|x-y|<t_k$ only if
\[(k+2)^{-1} < \dist(x,\partial \Omega) \leq (k-2)^{-1}.\]
Thus, $\phi_{t_k} \ast (\psi_k \vf)(x)$ is non-zero only if $x \in \Omega_{k+2}\bk \overline{\Omega}_{k-2}$. Hence, 
\[\supp(\phi_{t_k}\ast (\psi_k\vf)) \subseteq \Omega_{k+2}\bk \overline{\Omega}_{k-2} := B_k \Subset \Omega.\]
We will fix the precise value of each $t_k$ below.

Define
\[\vg(x) = \sum_{k=1}^\infty \phi_{t_k}\ast (\psi_k \vf)(x).\]
Since $\phi\in C_c^\infty(\Omega)$, each summand is in $C_c^\infty(\Omega;\R^d)$. Further, for each $x \in \Omega$, $x$ is in finitely many $B_k$, and so the series converges locally uniformly and $\vg\in C^\infty(\Omega;\R^d)$.

Fix $\epsilon>0$. We claim we can choose $t_k$ such that
$\|\vf-\vg\|_{\scrWpp(W,\Omega)}<\epsilon$. To prove this, we consider
each part of the norm \eqref{SobSpaceNorm} separately. Since
$W \in \calA_\pp$, by Theorem \ref{ConvolutionConvergence:theorem},
for each $k$, we may choose $s_k>0$ sufficiently small such that
$\|\psi_k \vf - \phi_{s_k}\ast (\psi_k\vf)\|_{\Lpp(W,\Omega)} <
\epsilon/(2^{k+1})$. Then, we have that
\begin{multline*}
  \|\vf-\vg\|_{\Lpp(W,\Omega)}
  = \bigg\| \sum_{k=1}^\infty
  \big(\psi_k\vf - \phi_{s_k}\ast (\psi_k \vf)\big)\bigg\|_{\Lpp(W,\Omega)}\\
	 \leq \sum_{k=1}^\infty \| (\psi_k \vf - \phi_{s_k}\ast(\psi_k \vf))\|_{\Lpp(W,\Omega)}
	\leq \sum_{k=1}^\infty \frac{\epsilon}{2^{k+1}}
	 =\frac{\epsilon}{2}.
       \end{multline*}
       
The argument for the second part of the norm \eqref{SobSpaceNorm} is
similar. We  will show for each $j=1,\ldots, n$, $\| \partial_j(\vf
-\vg)\|_{\Lpp(W,\Omega)} < \epsilon/(2n)$. Since $W\in \calA_\pp$, by
Theorem \ref{ConvolutionConvergence:theorem}, for each $k$, we may
choose $t_k$ sufficiently small such that $t_k\leq s_k$ and for each $j=1,\ldots, n$,
\[\| \partial_j(\psi_k \vf) - \phi_{t_k}\ast \partial_j(\psi_k \vf)\|_{\Lpp(W,\Omega)} < \frac{\epsilon}{n2^{k+1}}.\]

By Minkowski's inequality, for each $j=1,\ldots, n$,
\begin{multline*}
  \| \partial_j(\vf - \vg)\|_{\Lpp(W,\Omega)}
  = \bigg\| \sum_{k=1}^\infty
  \big(\partial_j(\psi_k \vf) -  \phi_{t_k}\ast \partial_j(\psi_k \vf)\big)\bigg\|_{\Lpp(W,\Omega)}\\
  \leq \sum_{k=1}^\infty
  \| \partial_j(\psi_k\vf) - \phi_{t_k} \ast \partial_j(\psi_k \vf)\|_{\Lpp(W,\Omega)}
  <\frac{\epsilon}{2n}.
\end{multline*}

Thus, we have shown that
\begin{equation*}
  \| \vf-\vg\|_{\scrWpp(W,\Omega)}
   = \| \vf-\vg\|_{\Lpp(W,\Omega)} + \sum_{j=1}^n \|  \partial_j(\vf - \vg)\|_{\Lpp(W,\Omega)} 
	 < \frac{\epsilon}{2} + \sum_{j=1}^n \frac{\epsilon}{2n} 
	 = \epsilon.
       \end{equation*}
       
Hence, $\scrWpp(W,\Omega) \subseteq \scrHpp(W,\Omega)$.
\end{proof}

\begin{remark}
  The proof of Theorem~\ref{H=W-scalar:theorem} is essentially the
  same as the proof of Theorem~\ref{H=W:theorem}, substituting scalar
  $f$ and $g$ for the vector functions $\vf$ and $\vg$.  The only difference
  is that to estimate the first part of the norm
  $\|f-g\|_{\Lpp(v,W,\Omega)}$--that is, $\|f-g\|_{\Lpp(v,\Omega)}$--we
    need to use the scalar version of
    Theorem~\ref{ConvolutionConvergence:theorem}, which is just this
    result when $d=1$.  The details are left to the reader.
  \end{remark}

  Finally, we prove Corollary~\ref{cor:smooth-dense-Rn}.

  \begin{proof}[Proof of Corollary~\ref{cor:smooth-dense-Rn}]
    We will prove this for $\scrWpp(W,\R^n)$; the proof 
    for $\scrWpp(v, W,\R^n)$ is essentially the same.
    Fix $\epsilon>0$ and $\vf \in \scrWpp(W,\R^n)$.  By
    Theorem~\ref{H=W:theorem}, there exists $\vh \in
    C^\infty(\R^n;\R^d) \cap \scrWpp(W,\R^n)$ such that
    $\|\vf-\vh\|_{\scrWpp(W,\R^n)}<\epsilon/2$.  Therefore, to complete
      the proof we need to find $\vg \in C_c^\infty(\R^n;\R^d)$ such
      that
      \[ \|\vg-\vh\|_{\scrWpp(W,\R^n)}<\epsilon/2. \]

        We prove this with a standard truncation argument.  For each
        $k\geq 2$, let $\nu_k \in C_c^\infty(\R^n)$ satisfy $0\leq
        \nu_k \leq 1$, $\nu_k(x)=1$ if $|x|<k$, $\nu_k(x)=0$ if
        $|x|>2k$, and $|\nabla \nu_k| \lesssim 1/k$.  Let $\vg_k =
        \nu_k\vh$.  Then $\vg_k \in C_c^\infty(\R^n; \R^d)$ and $\| \vg_k-\vh\|_{\Lpp(W,\R^n)} =
          \|(1-\nu_k)|W\vh|\|_{\Lpp(\R^n)}$.  
        Since $|W\vh|\in \Lpp(\R^n)$, by the dominated convergence
        theorem in the scale of variable Lebesgue spaces
        (see~\cite[Theorem~2.62]{cruz-uribe_variable_2013}), we have
        that
        \[ \lim_{k\rightarrow \infty} \| \vg_k-\vh\|_{\Lpp(W,\R^n)} = 0. \]
        Therefore, for all $k$ sufficiently large,
        \[ \| \vg_k-\vh\|_{\Lpp(W,\R^n)} < \frac{\epsilon}{2(n+1)}. \]

Now for $j=1,\ldots, n$, observe that
\[ \partial_j (\vg_k - \vh) = \partial_j (\nu_k \vh) - \partial_j \vh = \vh \partial_j \nu_k + \nu_k \partial_j \vh - \partial_j \vh.\]
Hence,
\begin{multline*}
\| \partial_j( \vg_k - \vh)\|_{\Lpp(W,\R^n)}
    \leq \| \vh \partial_j \nu_k\|_{\Lpp(W,\R^n)} + \| (1-\nu_k) |W \partial_j \vh| \|_{\Lpp(\R^n)} \\
    \leq \||\nabla \nu_k| \vh \|_{\Lpp(W,\R^n)} + \| (1-\nu_k) |W\partial_j \vh|\|_{\Lpp(\R^n)}.
\end{multline*}
Since $|\nabla \nu_k| \lesssim 1/k$, $\| |\nabla \nu_k|\vh\|_{\Lpp(W,\R^n)} \lesssim \frac{1}{k} \| \vh\|_{\Lpp(W,\R^n)}$, which converges to zero as $k\to \infty$. Also, by the same dominated convergence theorem argument as before, we have
\[\lim_{k\to \infty} \| (1-\nu_k) |W\partial_j\vh|\|_{\Lpp(\R^n)} =0.\]

Thus, for all $k$ sufficiently large, we have for all $j=1,\ldots, n$,
\[ \| \partial_j (\vg_k - \vh)\|_{\Lpp(W,\R^n)} < \frac{\epsilon}{2(n+1)}.\]
Therefore, we have shown for all sufficiently large $k$,
\begin{multline*}
\| \vg_k - \vh\|_{\scrWpp(W,\R^n)}
    = \|\vg_k - \vh\|_{\Lpp(W,\R^n)} + \sum_{j=1}^n \| \partial_j (\vg_k - \vh)\|_{\Lpp(W,\R^n)}\\
    < \frac{\epsilon}{2(n+1)} + \sum_{j=1}^n \frac{\epsilon}{2(n+1)} = \frac{\epsilon}{2}.
\end{multline*}
Hence, for $k$ sufficiently large,
\[\| \vf-\vg_k\|_{\scrWpp(W,\Omega)} \leq \| \vf-\vh\|_{\scrWpp(W,\Omega)} + \| \vh - \vg_k\|_{\scrWpp(W,\Omega)} <\frac{\epsilon}{2}+ \frac{\epsilon}{2} = \epsilon,\]
and so $C_c^\infty(\R^n;\R^d)$ is dense in $\scrWpp(W,\R^n)$.
  \end{proof}
  
\bibliography{MatrixApdotBib-MathSciNetFormat.bib}
\bibliographystyle{plain}

\end{document}